\documentclass{article}

\usepackage[centertags]{amsmath}
\usepackage{hyperref}
\usepackage{amsfonts}
\usepackage{amssymb}
\usepackage{amsthm}
\usepackage{newlfont}
\usepackage{amscd}
\usepackage{amsmath,amscd}

\usepackage[curve,arrow,cmactex,matrix]{xy}

\usepackage{verbatim}

\usepackage{color}
\usepackage{eucal}

\newcommand{\ZZ}{\mathbb{Z}}

\newcommand{\A}{\mathbb{A}}

\newcommand{\X}{\mathrm{X}}

\newcommand{\E}{\mathcal{E}}
\newcommand{\B}{\mathbf{B}}

\newcommand{\T}{\mathcal{T}}

\newcommand{\R}{\mathcal{R}}

\def\alp{{\alpha}}
\def\bet{{\beta}}
\def\gam{{\gamma}}

\def\sig{{\sigma}}
\def\vphi{{\varphi}}

\def\Gam{{\Gamma}}

\def\Lam{{\Lambda}}

\def\vphi{{\varphi}}

\def\Pro{\mathrm{Pro}}

\def\Set{\mathrm{Set}}

\newtheorem{thm}{Theorem}[section]

\newtheorem{lem}[thm]{Lemma}
\newtheorem{prop}[thm]{Proposition}

\theoremstyle{definition}
\newtheorem{define}[thm]{Definition}

\theoremstyle{remark}
\newtheorem{rem}[thm]{Remark}

\newtheorem{example}{Example}

\DeclareMathOperator{\Gal}{Gal}

\DeclareMathOperator{\Spec}{Spec}

\DeclareMathOperator{\precolim}{colim}

\def\colim{\mathop{\precolim}}

\DeclareMathOperator{\spec}{Spec}

\DeclareMathOperator{\Out}{Out}

\DeclareMathOperator{\Ob}{Ob}

\DeclareMathOperator{\fin}{fin}
\DeclareMathOperator{\Iso}{Iso}

\DeclareMathOperator{\df}{def}

\DeclareMathOperator{\Id}{Id}
\DeclareMathOperator{\FinAb}{FinAb}
\DeclareMathOperator{\FinSet}{FinSet}
\DeclareMathOperator{\Sec}{Sec}

\DeclareMathOperator{\Sch}{Sch}

\def\Aut{\textrm{Aut\,}}

\def\lrar{\longrightarrow}

\def\x{\stackrel}

\def \ovl{\overline}
\def \what{\widehat}
\def \wtl{\widetilde}

\title { The Section Conjecture for Graphs and Conical Curves }

\author{ Yonatan Harpaz }

\begin{document}
\maketitle

\begin{abstract}
In this paper we formulate and prove a combinatorial version of the section conjecture for finite groups acting on finite graphs. We apply this result to the study of rational points and show that finite descent is the only obstruction to the Hasse principle for mildly singular curves whose components are all geometrically rational.
\end{abstract}

\tableofcontents

\section{ Introduction }

The connection between rational points and sections of the Grothendieck exact sequence has an analogue for topological spaces carrying an action of a group. The corresponding section conjecture is false in general (as for general varieties), but it does hold for some special classes of spaces and groups. In this paper we will show that a certain variant of this section conjecture is true when the underlying space is a graph and the acting group is finite.

This result can be applied back to the study of rational points by considering algebraic varieties whose geometry is controlled by a graph. In particular, we will consider the following class of (mildly singular) curves:
\begin{define}
Let $K$ be a field with algebraic closure $\ovl{K}$. We will say that a projective reduced curve $C/K$ is \textbf{conical} if each irreducible component of $\ovl{C} = C \otimes_K \ovl{K}$ is $\ovl{K}$-rational (but not necessarily smooth). We will say that a projective reduced curve is \textbf{transverse} if each singular point is a transverse intersection of components, i.e. if for each singular point $s \in \ovl{C}\left(\ovl{K}\right)$ the branches of $\ovl{C}$ through $s$ intersect like $n$ coordinate axes at $0 \in \A^n$.
\end{define}

Conical curves were studies by A. Skorobogatov and the author in~\cite{sh}. In that paper it was shown that conical curves can posses surprising properties no smooth curve can ever have: they can contain infinitely many adelic points while having only finitely many rational points or none at all, and yet have a trivial Brauer group. In particular, the Brauer-Manin obstruction is \textbf{not} the only obstruction to the Hasse principle for such curves. This was used in~\cite{sh} in order to construct \textbf{smooth projective surfaces} with no rational points and infinite \'etale-Brauer sets. The surfaces in question were fibred into curves with some of the fibers being conical curves.

As the Brauer-Manin obstruction is not sufficient for conical curves, one is led to ask whether the \textbf{\'etale-Brauer} obstruction is sufficient. In other words, if a conical curve has a non-empty \'etale-Brauer set, must it have also a rational point? In this paper we will apply the section conjecture for graphs in order to prove the following result, implying in particular a positive answer to this question (at least in the transverse case):

\begin{thm}
\textbf{Finite descent} is the only obstruction to the Hasse principle for transverse conical curves over number fields.
\end{thm}

\begin{rem}
Since the obstruction of finite descent is weaker than the \'etale-Brauer obstruction this implies that the \'etale-Brauer obstruction is sufficient as well.
\end{rem}

\begin{rem}
Let $C$ be a curve. Then there exists a universal transverse curve $C'$ admitting a map $\rho:C' \lrar C$ (see \cite[p. 247]{blr}). The curve $C'$ can be obtained from the normalization $\wtl{C}$ of $C$ by gluing points which have the same image in $C$. Note that the map $\rho$ induce an isomorphism
$$ \rho_*:C'(L) \lrar C(K) $$
for every field $L$ containing $K$. In particular, when dealing with Diophantine questions such as the Hasse principle it is quite reasonable to restrict attention to the transverse case.
\end{rem}

This paper is organized as follows. In \S~\ref{s:grothendieck} we recall a general setting for the Grothendieck exact sequence and describe some relevant examples. In \S~\ref{s:pro-trees} we recall the notion of \textbf{pro-finite trees} as studied in~\cite{zm} and use it in order to prove an appropriate section conjecture for finite groups acting on finite graphs. In \S~\ref{s:inc-graphs} we explain how to model a conical curve $C$ by its \textbf{incidence graph} $X(C)$ and use it in order to relate the Grothendieck exact sequences of $C$ of $X(C)$. We will then use these results in order to prove the main theorem as stated above.

\section{ Grothendieck's short exact sequence }\label{s:grothendieck}
In this section we will recall Grothendieck's short exact sequence in a somewhat abstract setting. We do not claim any originality for the content of this section (most of it is essentially due to Grothendieck). See~\cite{d} \S 10 for a thorough presentation of the subject.

Let $\Pi$ be a small groupoid and $\Gam$ a group acting on $\Pi$. The Grothendieck construction associates with the pair $\Pi,\Gam$ a new groupoid $Q(\Pi,\Gam)$ as follows. The objects of $Q(\Pi,\Gam)$ are the objects of $\Pi$. Now for every two objects $x,y \in \Pi$ the morphisms $Q(\Pi,\Gam)(x,y)$ are pairs $(\sig,\alp)$ where $\sig \in \Gam$ and $\alp: \sig(x) \lrar y$ is a morphism in $\Pi$. Composition is given by
$$ (\sig,\alp) \circ (\tau,\bet) = (\sig\tau, \alp \circ \sig(\bet)) .$$
It is easy to verify that $Q(\Pi,\Gam)$ is indeed a groupoid. 

Let $\B\Gam$ denote the groupoid with one object $*$ such that $\B\Gam(*,*) = \Gam$. One then has a natural functor
$$ Q(\Pi,\Gam) \lrar \B\Gam $$
which sends all objects to $*$ and the morphism $(\sig,\alp)$ to the morphism $\sig$. The groupoid-theoretic fiber of this map can then be identified with $\Pi$.

\begin{rem}
This construction is a particular case of a more general Grothendieck construction which applies to the situation of a functor from a groupoid $\B$ to the category of small groupoids. The particular case above is obtained by taking $\B = \B\Gam$ in which case a functor from $\B$ to groupoids is just a groupoid with a $\Gam$-action.
\end{rem}

\begin{rem}
In homotopy theory the groupoid $Q(\Pi,\Gam)$ is also known as the \textbf{homotopy quotient} of $\Pi$ under the action of $\Gam$. This is why we denote it here with the letter $Q$, although this notation is not standard.
\end{rem}

Now assume that $\Pi$ is connected (i.e. every two objects are isomorphic) and let $x \in \Pi$ be an object. Let us denote by
$$ Q(\Pi,\Gam,x) \x{\df}{=} Q(\Pi,\Gam)(x,x) $$
the set of morphisms from $x$ to $x$ in $Q(\Pi,\Gam)$. We have a surjective homomorphism of groups
$$ Q(\Pi,\Gam,x) \lrar \B\Gam(x,x) = \Gam $$
whose kernel can be identified with $\Pi(x,x)$. In particular, we obtain a short exact sequence of the form
$$ 1 \lrar \Pi(x,x) \lrar Q(\Pi,\Gam,x) \lrar \Gam \lrar 1 $$
which we shall call the \textbf{Grothendieck exact sequence} associated with $(\Pi,\Gam)$.

Now suppose that $y \in \Pi$ is an object which is fixed by $\Gam$. One can then construct a section
$$ \xymatrix{
Q(\Pi,\Gam,x) \ar[r] & \Gam \ar@/_1pc/[l] \\
}$$
as follows. Choose an isomorphism $\vphi: x \lrar y$. Then for each $\sig \in \Gam$ we have an isomorphism
$$ \sig(\vphi): \sig(x) \lrar \sig(y)  = y .$$
Hence we can construct a section $s_\vphi: \Gam \lrar Q(\Pi,\Gam,x)$ by setting
$$ s_\vphi(\sig) = (\sig,\vphi^{-1} \circ \sig(\vphi)) .$$
It is not hard to verify that $s_\vphi$ is indeed a homomorphism of groups. Note that if we would have chosen a different isomorphism $\vphi': x \lrar y$ then we would get
$$ s_{\vphi'} = \psi^{-1} s_{\vphi} \psi $$
where
$$ \psi = \vphi^{-1} \circ \vphi' \in \Pi(x,x) \subseteq Q(\Pi,\Gam,x) .$$
Two sections which differ by a conjugation with an element of $\Pi(x,x)$ are called \textbf{conjugated}. We hence obtain a natural map from the set of $\Gam$-fixed objects of $\Pi$ to the set of \textbf{conjugacy classes of sections} of the form

$$ \xymatrix{
1 \ar[r] & \Pi(x,x) \ar[r] & Q\left(\Pi,\Gam,x\right) \ar[r] & \Gam \ar@/_1pc/[l] \ar[r] & 1 \\
}.$$

We will denote by $\Sec(\Pi,G)$ the set of conjugacy classes of sections as above and by
$$ h: \Pi^G \lrar \Sec(\Pi,G) $$
the map associating with each fixed object its corresponding conjugacy class of maps.

\begin{rem}\label{r:invariant-iso}
Two fixed objects $y,y'$ which are isomorphic via a \textbf{$\Gam$-invariant} isomorphism will give rise to the same conjugacy class. Hence the map $h$ always factors through a map
$$ \pi_0\left(\Pi^G\right) \lrar \Sec(\Pi,G) $$
where $\pi_0$ of a groupoid denotes the corresponding set of isomorphism types.
\end{rem}

\begin{example}\label{e:fundamental-groupoid}
Let $X$ be a topological space. The \textbf{fundamental groupoid} $\Pi_1(X)$ is the groupoid whose objects are the points of $X$ and the morphisms from $x$ to $y$ are the homotopy classes of continuous paths from $x$ to $y$. 

If $G$ is a group acting on $X$ then $\Pi_1(X)$ naturally inherits a $G$-action as well. The Grothendieck short exact sequence in this case is the sequence
$$ 1 \lrar \Pi_1(X,x) \lrar Q(\Pi_1(X),G,x) \lrar G \lrar 1 $$
where $\Pi_1(X,x)$ can be identified with the fundamental group of $X$ based at $x$. The section map then descends to a map
$$ \pi_0\left(X^G\right) \lrar \Sec(\Pi_1(X),G) .$$
In general this map is far from being an isomorphism. For example consider the group $\ZZ/2$ acting on $S_2$ via the antipodal involution. Then we have $\left(S^2\right)^{\ZZ/2} = \emptyset$ while $\Sec(\Pi_1(S_2),\ZZ/2) = \{*\}$ because $\Pi_1(S_2,x) = 1$ for every $x \in S^2$.
\end{example}

In order to describe the algebro-geometric analogue of Example~\ref{e:fundamental-groupoid} in the language of Grothendieck's construction it is useful to replace the notion of a groupoid by a \textbf{pro-finite} analogue:

\begin{define}
A \textbf{pro-finite groupoid} $\Pi$ is a groupoid enriched in pro-finite sets. More explicitly, we have:
\begin{enumerate}
\item
A collection of objects $\Ob(\Pi)$.
\item
For each $X,Y \in \Ob(\Pi)$ a \textbf{pro-finite set} of morphisms
$$ \Pi(X,Y) .$$
\item
For each three objects $X,Y,Z \in \Ob(\Pi)$ a \textbf{continuous} associative composition rule
$$ \Pi(X,Y) \times \Pi(Y,Z) \lrar \Pi(X,Z) .$$
\item
For each object $X \in \Ob(\Pi)$ a designated identity morphisms $\Id_X \in \Pi(X,X)$ which is neutral with respect to composition.
\item
For each morphism $f: X \lrar Y$ an inverse $f^{-1}: Y \lrar X$ (this is the \textbf{groupoid} part of the definition).
\end{enumerate}

\end{define}

\begin{rem}
The category of pro-finite sets can be naturally identified with the category of totally disconnected compact Hausdorf spaces. The reader is welcome to choose his or hers preferred point of view for this notion.
\end{rem}

Now let $\Pi$ be a pro-finite groupoid and let $\Gam$ be a pro-finite group acting on $\Pi$ (via continuous functors). The Grothendieck construction described above generalizes naturally to the pro-finite setting. In this case $Q(\Pi,\Gam)$ inherits a natural structure of a pro-finite groupoid and the Grothendieck short exact sequence
$$ 1 \lrar \Pi(x,x) \lrar Q(\Pi,\Gam,x) \lrar \Gam \lrar 1 $$
becomes a short exact sequence of pro-finite groups.

\begin{example}
Let $X$ be a variety over an \textbf{algebraically closed} field $k$. Let $\E(X)$ be the category whose objects are connected finite \'etale coverings $f:Y \lrar X$ and whose morphisms are maps over $X$. For each point $x \in X\left(k\right)$ we have an associated \textbf{fiber functor}
$$ F_x: \E(X) \lrar \Set $$
which sends $(f:Y \mapsto X) \in \E(X)$ to the fiber $f^{-1}\left(x\right) \in \Set$. Now given two points $x,y \in X\left(k\right)$ we can consider the set
$$ \Iso(F_x,F_y) $$
of \textbf{natural equivalences} from $F_x$ to $F_y$. More explicitly, each $\alp \in \Iso(F_x,F_y)$ consists of a compatible family of isomorphisms
$$ \alp_f: f^{-1}(x) \lrar f^{-1}(y) $$
indexed by the objects $f: Y \lrar X$ of $\E(X)$. The set $\Iso(F_x,F_y)$ carries a natural topology in which the neighborhood basis of an element $\alp \in \Iso(F_x,F_y)$ consists of the sets
$$ W_{\alp,f} = \left\{\bet \in \Iso(F_x,F_y) | \bet_f = \alp_f\right\} $$
for each $f: Y \lrar X$ in $\E(X)$. This topology makes $\Iso(F_x,F_y)$ into a totally disconnected compact Hausdorf space, or in other words, a pro-finite set. Since natural equivalences can be composed and inverted we get a pro-finite groupoid $\Pi_1\left(X\right)$ whose objects are
$$ \Ob\left(\Pi_1\left(X\right)\right) = X\left(k\right) $$
and whose morphisms are
$$ \Pi_1\left(X\right)(x,y) \x{\df}{=} \Iso(F_x,F_y) .$$
This pro-finite groupoid is called the \textbf{\'etale fundamental groupoid} of $X$. Given a point $x \in X\left(k\right)$ one denotes
$$ \Pi_1\left(X,x\right) \x{\df}{=} \Pi_1(X)(x,x) .$$
This pro-finite group is called the \textbf{\'etale fundamental group} of $X$ based at $x$.
\end{example}

If $X$ is a variety over a general field $K$ and $\ovl{X} = X \times_K \ovl{K}$ is the base change to the algebraic closure then $\Pi_1\left(\ovl{X}\right)$ inherits a natural action of $\Gam_K = \Gal\left(\ovl{K}/K\right)$ by (continuous) functors. Taking the associated Grothendieck exact sequence we obtain a short exact sequence of pro-finite groups
$$ 1 \lrar \Pi_1\left(\ovl{X},x\right) \lrar Q\left(\Pi_1\left(\ovl{X}\right),\Gam_K,x\right) \lrar \Gam_K \lrar 1 .$$
Note that any $K$-rational point of $X$ determines a $\Gam_K$-invariant object of $\Pi_1\left(\ovl{X}\right)$ and hence a conjugacy class of sections
$$ \xymatrix{
Q\left(\Pi_1\left(\ovl{X}\right),\Gam_K,x\right) \ar[r] & \Gam_K \ar@/_1pc/[l] \\
}.$$

\begin{example}\label{e:spaces}
Let $X$ be a topological space and $G$ a finite group acting on $X$. Imitating the algebraic construction of the fundamental groupoid one can consider the category $\E(X)$ of finite covering maps $Y \lrar X$ and consider the groupoids $\what{\Pi}_1(X)$ whose objects are the points of $X$ and such that the morphisms from $x$ to $y$ are the natural transformations between the associated fiber functors (carrying their natural pro-finite topology). There is a natural functor
$$ \Pi_1(X) \lrar \what{\Pi}_1(X) $$
which identifies $\what{\Pi}_1(X)$ with the \textbf{pro-finite completion} of $\Pi_1(X)$ (in an appropriate sense). In particular, for each $x \in X$ the group $\what{\Pi}_1(X,x)$ is the pro-finite completion of the fundamental group $\Pi_1(X,x)$.
\end{example}

We will be interested in a variant of example~\ref{e:spaces} with \textbf{graphs} instead of spaces. To fix notation, let us recall what we mean by a \textbf{graph}:
\begin{define}\label{d:graph}
A \textbf{graph} $X = (X_0,X_1,s,t)$ is a pair of sets $X_0,X_1$ (called vertices and edges respectively) together with two maps
$$ s,t: X_1 \lrar X_0 $$
associating with each edge its source and target respectively. Note that every edge carries an orientation (i.e. a graph for us is always \textbf{directed}) and that both loops and multiple edges are allowed.
\end{define}
\begin{define}\label{d:graph-covering}
A map $f:Y \lrar X$ of graphs is called a \textbf{graph covering} if the following two conditions are satisfied:
\begin{enumerate}
\item 
For each edge $e \in X_1$ and every vertex $v \in Y_1$ such that $f(v) = s(e)$ there exists a unique edge $\wtl{e} \in Y_1$ such that $f(\wtl{e}) = e$ and $s(\wtl{e}) = v$.
\item
For each edge $e \in X_1$ and every vertex $v \in Y_1$ such that $f(v) = t(e)$ there exists a unique edge $\wtl{e} \in Y_1$ such that $f(\wtl{e}) = e$ and $t(\wtl{e}) = v$.
\end{enumerate}
\end{define}

A covering map $Y \lrar X$ is called \textbf{finite} if the preimage of each vertex $v \in X_1$ is finite. We will denote by $\E(X)$ the category of finite coverings maps $Y \lrar X$. 

Now recall that every graph $X$ can be geometrically realized into a topological space $|X|$. It is not hard to see that a map $Y \lrar X$ of graphs is a graph covering if and only if the induced map $|Y| \lrar |X|$ is a covering map of topological spaces. Furthermore, every covering map of $|X|$ arises this way. Hence we can describe the pro-finite fundamental groupoid $\what{\Pi}_1(|X|)$ in terms of $X$. More explicitly, given a graph $X$ and vertex $v \in V(X)$ we can consider the fiber functor
$$ F_v: \E(X) \lrar \Set $$
sending each $p: Y \lrar X$ in $\E(X)$ to the set $p^{-1}(v)$. We can then consider the combinatorial version $\what{\Pi}_1(X)$ of $\what{\Pi}_1(|X|)$ to be the pro-finite groupoid whose objects are the vertices of $X$ and whose morphisms are the natural transformations between the associated fiber functors. The realization functor then induces an equivalence of pro-finite groupoids
$$ \what{\Pi}_1(X) \x{\simeq}{\lrar} \what{\Pi}_1(|X|) .$$

Now let $G$ be a group acting on a graph $X$. We then have an induced action of $G$ on $\what{\Pi}_1(X)$ and we can consider the associated Grothendieck exact sequence
$$ 1 \lrar \what{\Pi}_1(X,x) \lrar Q\left(\what{\Pi}_1(X),G,x\right) \lrar G \lrar 1 .$$
The section map then descends to a map
$$ \pi_0\left(X^G\right) \lrar \Sec(\Pi_1(X),G) $$
where $\pi_0$ of a graph denotes the set of connected components of the graph (which can be identified with the set of connected components of the geometric realization of the graph). In \S~\ref{s:pro-trees} below we will prove that when $X$ and $G$ are finite then this section map is an \textbf{isomorphism}. We consider this result as the combinatorial analogue of Grothendieck's section conjecture for curves.

\section{ Pro-finite trees and the section conjecture for graphs }\label{s:pro-trees}
The purpose of this section is to prove the \textbf{section conjecture for graphs} (see Theorem~\ref{t:sec-conj} below). For this purpose we will recall the theory of \textbf{pro-finite trees} as developed in~\cite{zm}.

\begin{define}
A \textbf{pro-finite graph} is a graph $X = (X_0,X_1,s,t)$ (see Definition~\ref{d:graph}) together with a pro-finite topology $\T$ on the set $X_0 \coprod X_1$ such that
\begin{enumerate}
\item
$X_0 \subseteq X_1 \coprod X_0$ is closed in $\T$. In particular the induced topology on $X_0$ is pro-finite as well.
\item
The induced maps
$$ (s)_*, (t)_* : X_1 \coprod X_0 \lrar X_0 $$
(which are the identity on the second component) are continuous.
\end{enumerate}
\end{define}

\begin{example}
Every finite graph can be considered as a pro-finite graph in a unique way.
\end{example}

\begin{example}
Let $\Gam$ be a pro-finite group and $S \subseteq \Gam$ a subset not containing $1$. Let $(\Gam,\Gam \times S,s,t)$ be the Cayley graph of $\Gam$ with respect to $S$ (so that $s(\gam,a) = \gam, t(\gam,a) = \gam\cdot a$ for $\gam \in \Gam,a\in S$). Then $(\Gam,\Gam \times S,s,t)$ carries a natural structure of a pro-finite graph by endowing
$$ \left[\Gam \times S\right] \coprod \Gam \cong \Gam \times \left[S \cup \{1\}\right] $$
with the product topology. We call this graph the \textbf{pro-finite Cayley graph} of $\Gam$.
\end{example}

\begin{define}
We will say that a pro-finite graph $(X_0,X_1,s,t,\T)$ is \textbf{split} if $X_1 \subseteq X_1 \coprod X_0$ is closed in $\T$. In this case the induced topology on $X_1$ is pro-finite as well and $\T$ exhibits $X_1 \coprod X_0$ as a topological disjoint union of $\left(X_1,\T|_{X_1}\right)$ and $\left(X_0,\T|_{X_0}\right)$.
\end{define}

\begin{example}
Let $\Gam$ be a pro-finite set and $S \subseteq \Gam$ a subset not containing $1$. Let $X = (\Gam,\Gam \times S,s,t,\T)$ be the pro-finite Cayley graph of $\Gam$ with respect to $S$. Then $X$ is split if and only if $1$ is not in the closure of $S$ (e.g. when $S$ is finite).
\end{example}

All the pro-finite graphs appearing in this paper will be split (in fact, they will be very close to the pro-finite Cayley graph of a free pro-finite group on a finite set of generators). Since this restriction simplifies things somewhat we will restrict our attention to split pro-finite graphs from now on.

Let $\Pro(\FinAb)$ be the category of pro-finite abelian groups and $\Pro(\FinSet)$ the category of pro-finite sets. The forgetful functor
$$ U:\Pro(\FinAb) \lrar \Pro(\FinSet) $$
commutes with all limits and so admits a left adjoint
$$ \what{\ZZ}:\Pro(\FinSet) \lrar \Pro(\FinAb) .$$
If $X$ is a pro-finite set then the pro-finite abelian group $\what{\ZZ}(X)$ can be considered as the pro-finite abelian group (or $\what{\ZZ}$-module, justifying the notation) \textbf{freely generated from $X$}. The unit map $\iota: X \lrar \what{\ZZ}(X)$ is an embedding and if $A$ is a pro-finite abelian group then continuous homomorphisms
$$ \what{\ZZ}(X) \lrar A $$
are in one-to-one correspondence (obtained via $\iota$) with maps of pro-finite sets
$$ X \lrar U(A) .$$

\begin{rem}
Note that if $X$ is not discrete (i.e. not finite) then $\what{\ZZ}(X)$ will generally not be free in the classical (pro-finite) sense. In fact, one has 
$$ \what{\ZZ}(X) \cong \lim_\alp \what{\ZZ}(X_\alp) $$
where $X_\alp$ runs over all the finite quotients of $X$.
\end{rem}

Now let $X=(V,E,s,t,\T)$ be a split pro-finite graph. We have an associated complex of pro-finite abelian groups
$$ \what{\ZZ}(E) \x{\partial_1}{\lrar} \what{\ZZ}(V) \x{\partial_0}{\lrar} \what{\ZZ} $$
where $\partial_1(e) = t(e) - s(e)$ for every $e \in E \subseteq \what{\ZZ}(E)$ and $\partial_0(v) = 1$ for every $v \in V \subseteq \what{\ZZ}(V)$. One then defines the \textbf{reduced homology groups} $\wtl{H}_1(X),\wtl{H}_0(X)$ to be the homology groups of the above complex. These groups carry a natural structure of pro-finite abelian groups.

\begin{example}
If $X$ is a finite graph then $\wtl{H}_0(X)$ and $\wtl{H}_1(X)$ coincide with the reduced homology groups of $X$ with coefficients in $\what{\ZZ}$ (in the usual topological sense).
\end{example}

The following lemma appears in~\cite{zm}:
\begin{lem}[~\cite{zm} (1.7, 1.12)]\label{l:zm}
Let $X$ be a pro-finite graph. Then for $i=0,1$ one has
$$ \wtl{H}_i(X) \cong \lim_\alp \wtl{H}_i(X_\alp) $$
where $X_\alp$ runs over the finite quotients of $X$.
\end{lem}

We say that a pro-finite graph $X$ is connected if $\wtl{H}_0(X) = 0$. Note that if $X$ is a finite graph then this notion coincides with the usual notion of connectedness. Using Lemma~\ref{l:zm} we see that a pro-finite graph is connected if and only if each finite quotient graph of $X$ is connected. However, note that $X$ might be connected as a pro-finite graph and yet \textbf{not connected} as a discrete graph (i.e. when forgetting the pro-finite topology).

\begin{define}
Let $X$ be a pro-finite graph. $X$ is said to be a \textbf{pro-finite tree} if it is connected and $\wtl{H}_1(X) = 0$.
\end{define}

We will be interested in the following types of pro-finite trees. Let $X$ be a finite connected graph and let $p:\wtl{X} \twoheadrightarrow X$ be a universal cover. Let $A = \Aut_X\left(\wtl{X}\right)$ be the group of automorphisms of $\wtl{X}$ over $X$. For each finite index subgroup $H < A$ the (connected) quotient graph $Y_H \x{\df}{=} \wtl{X}/H$ admits a natural covering map $Y_H \lrar X$. We define the \textbf{pro-universal cover} $\what{X}$ of $X$ to be the limit (in the category of pro-finite graphs):
$$ \what{X} = \lim_{H < A, [A:H] < \infty} Y_H .$$
Then by Lemma~\ref{l:zm} we get that $\what{X}$ is a connected. Furthermore it is easy to see that for each $H < A$ and each $u \in \wtl{H}_1(Y_H)$ there exists an $H' < H$ of finite index such that $u$ is \textbf{not} in the image of the induces map
$$ \wtl{H}_1\left(Y_{H'}\right) \lrar \wtl{H}_1(Y_H) .$$
Hence $\wtl{H}_1\left(\what{X}\right) = 0$ and so $\what{X}$ is a \textbf{pro-finite tree}.

The main result we will need from~\cite{zm} is the following fixed point theorem. It is a known fact that a finite group acting on an (oriented) graph always has a fixed vertex. This is usually considered as part of the comprehensive Bass-Serre theory concerning groups acting on trees. The following is a generalization of this fixed point theorem to the pro-finite setting:
\begin{thm}[~\cite{zm}]\label{t:fixed}
Let $G$ be a finite group acting on a pro-finite tree $T$. Then there exists a vertex $v$ fixed by $G$. In fact, the fixed subgraph $T^G$ is also a pro-finite tree.
\end{thm}

Our main application of this theorem is the following result, which can be considered as the \textbf{section conjecture for finite graphs}. Let $X$ be a finite graph acted on by a finite group $G$. Recall the \textbf{pro-finite fundamental groupoid} $\what{\Pi}_1(X)$ which was defined in \S~\ref{s:inc-graphs}. Then $\what{\Pi}_1(X)$ carries a $G$-action and we have the Grothendieck exact sequence
$$ 1 \lrar \what{\Pi}_1(X,v) \lrar Q(\what{\Pi}_1(X),G,v) \lrar G \lrar 1 .$$
Recall that $\Sec\left(\what{\Pi}_1(X),G\right)$ denotes the set of conjugacy classes of sections
$$ \xymatrix{
Q\left(\what{\Pi}_1(X),G,v\right) \ar[r] & G \ar@/_1.5pc/[l] \\
}$$
and that each fixed vertex of $X$ gives an element of $\Sec\left(\what{\Pi}_1(X),G\right)$, i.e. we have a map
$$ h: V\left(X^G\right) \lrar \Sec\left(\what{\Pi}_1(X),G\right) .$$
We then claim the following:
\begin{thm}\label{t:sec-conj}
Let $X$ be a finite graph acted on by a finite group $G$. Then the map
$$ h: X^G \lrar \Sec\left(\what{\Pi}_1(X),G\right) $$
is surjective an induces a bijection
$$ h: \pi_0\left(X^G\right) \lrar \Sec\left(\what{\Pi}_1(X),G\right) $$
where $\pi_0$ of a graph denotes the corresponding set of connected components.
\end{thm}
\begin{proof}
Let $A = \Aut\left(\wtl{X}/X\right)$ as above and
$$ \what{X} = \lim_{H < A, [A:H] < \infty} Y_H $$
be the pro-universal cover of $X$ with $p:\what{X} \lrar X$ the induced map. Let $\what{v} \in \what{X}$ be a vertex over $v$. Our first observation is that the choice of $\what{v}$ determines an action of $Q\left(\what{\Pi}_1(X),G_L,v\right)$ on $\what{X}$. Although this is a straight forward  pro-finite adaptation of standard covering space theory let us explain this point with some more detail as to establish the setting for the rest of the proof.

First note that inside the partially ordered set (poset) of finite index subgroups $H < A$ the sub-poset of finite index \textbf{normal} subgroups is cofinal. Hence we have a natural isomorphism
$$ \what{X} \cong \lim_{H \triangleleft A, [A:H] < \infty} Y_H .$$
Let $\alp \in \what{\Pi}_1(X,v)$ be an element (so that $\alp: v \lrar v$ is a morphism in $\what{\Pi}(X)$) and let $\what{u} = \alp^{-1}_*\left(\what{v}\right)$ be the the pre-image of $\what{v}$ under the induced map
$$ \alp_*: p^{-1}(v) \lrar p^{-1}(v) .$$
For each $H \triangleleft A$ of finite index let $v_H,u_H$ be the respective images of $\what{v},\what{u}$ in $Y_H = Y/H$. Since $H$ is normal we have that $Y/H \lrar X$ is a normal covering and so there exists a unique automorphism $T_{\alp,H} \in \Aut_X(Y_H) = A/H$ which sends $v_H$ to $u_H$. The automorphisms $T_{\alp,H}$ are compatible as $H$ ranges over finite index normal subgroups and so induce a pointed automorphism
$$ \xymatrix{
\left(\what{X},\what{v}\right) \ar^{T_\alp}[rr]\ar[dr] && \left(\what{X},\what{u}\right) \ar[dl] \\
& \left(X,v\right) & \\
}$$
The association $\alp \mapsto T_\alp$ is easily seen to be a free action of $\what{\Pi}_1(X,v)$ on $\what{X}$. In fact, this identifies $\what{\Pi}_1(X,v)$ with the group
$$ \Aut_X\left(\what{X}\right) \cong \lim_{H \triangleleft A, [A:H] < \infty} \Aut_X(Y_H) = \lim_{H \triangleleft A, [A:H] < \infty} A/H =  \what{A} .$$
Note that in order to make this identification we needed to choose the vertex $\what{v}$, but different choices would have resulted in conjugated isomorphisms. In particular, we have a canonical identification of finite index normal subgroups of $\what{\Pi}(X,v)$ and finite index normal subgroups of $\what{A}$, which in turn is the same as finite index normal subgroups of $A$.

Let us now explain how to extend this action to all of $Q\left(\what{\Pi}_1(X),G,v\right)$. The quotient
$$ Q\left(\what{\Pi}_1(X),G,v\right) \twoheadrightarrow G $$
induces a \textbf{quasi-action} of $G$ on $\what{\Pi}_1(X)$, i.e. a homomorphism $G \lrar \Out\left(\what{\Pi}_1(X,v)\right)$. In light of the above discussion this gives a quasi-action of $G$ on $\what{A}$ which in turn gives an action of $G$ on the poset of finite index normal subgroups in $A$. Since $G$ is finite we see that the sub-poset of $G$-invariant normal subgroups $H \triangleleft A$ is cofinal, and hence we can assume that the limit
$$ \what{X} = \lim_{H \triangleleft A, [A:H] < \infty} Y_H $$
is taken only over $G$-invariant $H$'s. Now given an element $(g,\alp) \in Q\left(\what{\Pi}_1(X),G,v\right)$ (so that $\alp: g(v) \lrar v$ is a morphism in $\what{\Pi}(X)$) we set as before $\what{u} = \alp^{-1}_*\left(\what{v}\right) \in \what{X}$ (note that this time $\what{u}$ lies above $g(v)$). For each $G$-invariant finite index $H \triangleleft T$ we let $v_H,u_H$ be the respective images of $\what{v},\what{u}$ in $Y_H$. Then there exists a unique automorphism of graphs $T_{g,\alp,H}: Y/H \lrar Y/H$ fitting into a commutative diagram of pointed graphs
$$ \xymatrix{
\left(Y/H,v_H\right) \ar^{T_{g,\alp,H}}[r]\ar[d] & \left(Y/H,u_H\right) \ar[d] \\
\left(X,v\right) \ar^{g}[r] & \left(X,g(v)\right) \\
}$$
As above the maps $T_{g,\alp,H}$ are compatible as $H$ ranges over $G$-invariant finite index normal subgroups $H \triangleleft A$. Hence we get a map $ T_{\alp,g}: \what{X} \lrar \what{X} $ which fits in the commutative diagram
$$ (*)\;\; \xymatrix{
\left(\what{X},\what{v}\right) \ar^{T_{g,\alp}}[r]\ar[d] & \left(\what{X},\what{u}\right) \ar[d] \\
\left(X,v\right) \ar^{g}[r] & \left(X,g(v)\right) \\
}$$
It is not hard to see that the association $(g,\alp) \mapsto T_{g,\alp}$ gives an action of $Q\left(\what{\Pi}(X),G,v\right)$ on $\what{X}$ extending the action of $\what{\Pi}_1(X,v)$ and covering the action of $G$ on $X$.

Let us now proceed with the proof of Theorem~\ref{t:sec-conj}. For each section $s: G \lrar Q(\what{\Pi}_1,G)$ we can restrict the action of $Q(\what{\Pi}_1(X),G)$ on $\what{X}$ to $G$ via $s$. According to Theorem~\ref{t:fixed} the fixed subgraph
$$ \what{X}^{s(G)} \subseteq \what{X} $$
is a \textbf{tree}, and in particular non-empty. Furthermore, if $s,s'$ are two \textbf{different} sections then
$$ \what{X}^{s(G)} \cap \what{X}^{s'(G)} = \emptyset $$
because $s(g)^{-1}s'(g) \in \what{\Pi}_1(X,v)$ which acts freely on $\what{X}$. Now if two sections $s,s'$ are \textbf{conjugated} by $\gam \in \what{\Pi}_1(X,v)$ then
$$ \gam\left(\what{X}^{s(G)}\right) = \what{X}^{s'(G)} .$$
However, if two sections $s,s'$ are \textbf{not} conjugated then by the above considerations we will get that
$$ \gam\left(\what{X}^{s(G)}\right) \cap \what{X}^{s'(G)} = \emptyset $$
for every $\gam \in \what{\Pi}_1(X,v)$. Hence the images
$$ \left\{p\left(\what{X}^{s(G)}\right)\right\}_{[s] \in \Sec\left(\what{\Pi}_1(X),G\right)} $$
is a \textbf{pairwise disjoint} collection of connected subgraphs of $X^G$.

Now let $w \in X^G$ be a fixed vertex and let $\vphi: v \lrar w$ be a morphism in $\what{\Pi}_1(X)$. Then $h(w) \in \Sec\left(\what{\Pi}_1,G\right)$ is represented by the section
$$ s_w: g \mapsto (g,\alp) \in Q(\what{\Pi}_1,X,v) $$
where $\alp = \vphi^{-1} \circ g(\vphi)$.

Let $\what{u} = \alp^{-1}_* \in p^{-1}(g(v))$ as above and let $\what{w} = \vphi_*\left(\what{v}\right) \in p^{-1}(w)$ be the image of $\what{v}$ under the bijection
$$ \vphi_*: p^{-1}(v) \x{\simeq}{\lrar} p^{-1}(w) .$$
Now recall that for every finite covering $q:Y \lrar X$ considered as an object $q \in \E(X)$ we have $g(q) = g \circ q$ and so by definition the map
$$ (g(\vphi))_{q}: q^{-1}(g(v)) \lrar q^{-1}(g(w)) $$
is equal to the map
$$ \vphi_{q'}: q'^{-1}(v) \lrar q'^{-1}(w) $$
where $q' = g^{-1}(q) = g^{-1} \circ q$. In view of the commutative diagram $(*)$ this implies that the following diagram of pointed pro-finite sets

$$ \xymatrix{
\left(p^{-1}(v),\what{v}\right) \ar^{(F_{g,\alp})_*}[d]\ar^{\vphi_*}[r] & \left(p^{-1}(w),\what{w}\right) \ar^{(F_{g,\alp})_*}[d] \ar@{=}[r] & \left(p^{-1}(w),\what{w}\right) \ar^{(F_{g,\alp})_*}[d] \\
\left(p^{-1}(g(v)),\what{u}\right) \ar^{g(\vphi)_*}[r] & \left(p^{-1}(g(w)),\what{w}\right) \ar@{=}[r] & \left(p^{-1}(w),\what{w}\right) \\
}$$
is commutative, and so $\what{w}$ is \textbf{fixed} by $F_{g,\alp}$ for every $g$, i.e.
$$ w \in p\left(\what{X}^{s_w(G)}\right) .$$
Since the various $\what{X}^{s(G)}$'s are pairwise disjoint it follows that for each $[s] \in \Sec\left(\what{\Pi}_1(X),g\right)$ one has
$$ h^{-1}([s]) = V\left(p\left(\what{X}^{s(G)}\right)\right) \neq \emptyset .$$
Since $p\left(\what{X}^{s(G)}\right)$ is also connected (because $\what{X}^{s(G)}$ is a tree) we get that all the vertices in $h^{-1}([s])$ sit in the same connected component of $X^G$. On the other hand, by using Remark~\ref{r:invariant-iso} it is not hard to show that vertices in the same connected component of $X^G$ have identical images under $h$. This means that $h$ induces a bijection between the connected components of $X^G$ and the elements of $\Sec\left(\what{\Pi}_1(X),g\right)$.
\end{proof}

\begin{rem}
The statement of Theorem~\ref{t:sec-conj} remains true if one replaces the pro-finite fundamental groupoid $\what{\Pi}_1(X) \cong \what{\Pi}_1(|X|)$ with the discrete groupoid $\Pi_1(|X|)$. The proof is essentially the same (using the classical fixed point theorem for finite groups acting on trees instead of Theorem~\ref{t:fixed}). The reason we focused our attention on the pro-finite case was our desired applications for rational points on curves (see \S~\ref{s:inc-graphs}).
\end{rem}

\section{ Finite descent on conical curves}\label{s:inc-graphs}

In this section we will turn our attention to \textbf{transverse conical curves} over a \textbf{number field} $K$. The main purpose of this section is to prove Theorem~\ref{t:main} below stating that finite descent is the only obstruction to the Hasse principle for transverse conical curves. 

We begin by encoding the structure of a conical curve in a graph. The following construction and notation recalls that of~\cite{sh}. Let $K$ be a number field with algebraic closure $\ovl{K}$ and let $C/K$ be a projective reduced curve with a normalization map $\nu: \wtl C \lrar C$. Consider the $0$-dimensional schemes
$$ \Lambda = \Spec\left(H^0\left(\wtl C,\mathcal{O}_{\wtl C}\right)\right), \quad
\Pi = C_{\mathrm{sing}}, \quad
\Psi = \nu^{-1}\left(C_{\mathrm{sing}}\right) .$$
Here $\Lambda$ is the $K$-scheme of irreducible components of $C$ (or the connected components of $\wtl C$) and $C_{\mathrm{sing}}$ is the singular locus of $C$. Let $\ovl\Lambda$, $\ovl \Pi$, $\ovl \Psi$ be the $\ovl K$-schemes obtained from $\Lambda$, $\Pi$, $\Psi$ by extending the ground field to $\ovl{K}$. Note that $0$-dimensional schemes over an algebraically closed field can be naturally identified with sets and so we can think of $\ovl\Lambda, \ovl\Pi$ and $\ovl\Psi$ as the sets of irreducible components of $\ovl C$, singular points of $\ovl C$ and critical points of $\ovl{\nu}$ respectively. In addition, since these schemes were base changed from $K$-schemes the corresponding sets carry an action of the Galois group $\Gam_K$.

\begin{define}
Let $C/K$ be a projective reduced curve. We define the {\bf incidence graph} $X(C)$ of $C$ to be the graph whose vertex set is
$$ X(C)_0=\ovl \Lambda\cup\ovl \Pi $$
and whose set of edges is $X(C)_1=\ovl \Psi$. The source map
$$ s: \ovl\Psi \lrar \ovl\Lambda \coprod \ovl\Pi $$
associates with each $Q \in \ovl \Psi$ the connected component $L\in \ovl\Lambda$ containing it and the target map
$$ t: \ovl\Psi \lrar \ovl\Lambda \coprod \ovl\Pi $$
associates with $Q \in \ovl{\Psi}$ its image $P = \nu(Q) \in \ovl \Pi$. By construction $X(C)$ is a bipartite graph with a natural action of the Galois group $\Gamma_K$.
\end{define}

\begin{rem}\label{r:functoriality}
The construction $C \mapsto X(C)$ is not, strictly speaking, functorial. However, it is functorial if we restrict ourselves to \textbf{\'etale maps} $C \lrar D$ of curves. Furthermore, given a finite \'etale map $C \lrar D$ the resulting map of graphs $X(C) \lrar X(D)$ is a finite graph covering (see Definition~\ref{d:graph-covering}). In particular, given a reduced projective curve $C$ we obtain a functor
$$ \X:\E\left(\ovl{C}\right) \lrar \E(X(C)) .$$

\end{rem}

\begin{rem}
If $C$ is a reduced projective curve and $\rho: C' \lrar C$ is the universal map from a transverse curve then $\rho$ induces an isomorphism of graphs
$$ \rho_*: X(C') \lrar X(C) .$$
\end{rem}

\begin{define}
Let $C/K$ be a reduced projective curve. We define the \textbf{splitting field} $L/K$ to be the minimal Galois extension which splits $\ovl\Lam,\ovl\Pi$ and $\ovl\Psi$. We will denote by
$$ G_L = \Gal(L/K) $$
the Galois group of this finite extension. Note that the action of $\Gam_K$ on $X(C)$ factors through the quotient
$$ \Gam_K \twoheadrightarrow G_L  .$$
\end{define}

Our first lemma verifies that when $C$ is \textbf{transverse conical curve} then the incidence graph of $X(C)$ captures all the information regarding finite \'etale coverings of $\ovl{C}$. 

\begin{prop}\label{p:equivalence}
Let $C/K$ be a transverse conical curve. Then the functor (see Remark~\ref{r:functoriality}) 
$$ \X: \E(\ovl{C}) \lrar \E(X(C)) $$
is an equivalence of categories.
\end{prop}
\begin{proof}

In order to show that $\X$ is an equivalence we will construct an explicit inverse for it. Recall that $X(C)$ is a \textbf{two-sided} graph - the edges always go from the subset $\ovl{\Pi}$ to the subset $\ovl{\Lam}$. This property is inherited by any covering $Y \lrar X(C)$: one can just divide the vertices of $Y$ to the preimage of $\ovl{\Pi}$ and the preimage of $\ovl{\Lam}$.

Note any two-sided graph $Y$ can be considered as a category with the vertices being the objects and each edge interpreted as a morphism from its source to its target (the identity morphisms are added formally). Due to the two-sidedness no pair of non-trivial morphisms is composable and so we don't need to add any new morphisms.

Now the structure of $X(C)$ as the incidence graph of $C$ defines a natural functor
$$ \mathcal{S}: X \lrar \Sch/_{L} $$
from $X$ to schemes over $L$ where the vertices in $\ovl{\Pi}$ map to their corresponding $L$-point (considered as a copy of $\spec(L)$) and the vertices in $\ovl{\Lam}$ map to the corresponding component (considered as a rational curve over $L$).

Given a covering map $p: Y \lrar X(C)$ we can consider it as a functor from $Y$ to $X(C)$ and hence obtain a functor
$$ \mathcal{S} \circ p: Y \lrar \Sch/_{L} .$$
We then define a functor
$$ \R_L: \E(X(C)) \lrar \E\left(C_L\right) $$
by associating with each $p: Y \lrar X(C)$ in $\E(X(C))$ the curve
$$ \R_L(Y) \x{\df}{=} \colim_{v \in Y} \mathcal{S}(p(v)) $$
which admits a natural finite \'etale map to $\ovl{C}_L$. We call $\R_L(p)$ the $L$-realization of the covering map $p: Y \lrar X(C)$. Now composing with the base change functor
$$ - \otimes_L \ovl{K}: \E\left(C_L\right) \lrar \E\left(\ovl{C}\right) $$
we obtain a functor
$$ \R: \E(X(C)) \lrar \E\left(\ovl{C}\right) .$$
We have natural maps
$$ u: D \lrar \R(\X(D)) $$
and
$$ v: \X(\R(Y) \lrar Y $$
which are easily seen to be isomorphisms. Hence the functor $\R$ is an inverse to $\X$ and we are done.
\end{proof}

\begin{rem}
Note that if $C/K$ is a transverse conical curve and $Y \lrar X(C)$ is a covering of graphs then $\R(Y)$ is a transverse conical curve as well. Hence we get that every \'etale covering of a transverse conical curve is conical.
\end{rem}

We are now in a position to relate the Grothendieck exact sequences of associated with $\ovl{C}$ and $X(C)$. A finite \'etale map of the form $\R(Y) \lrar \ovl{C}$ carries by construction a canonical $L$-structure. This means that the action of $\Gam_K$ on $\E\left(\ovl{C}\right)$ essentially factors over $L$. More precisely, the $G_L$ action on $X(C)$ induces an action of $G_L$ on $\E(X(C))$ such that $\sig \in G_L$ sends $p: Y \lrar X(C)$ to $\sig \circ p: Y \lrar X(C)$. If we pull this action back to an action of $\Gam_K$ we will get that the functor $\R$ is $\Gam_K$-equivariant (more precisely, it carries a natural structure of $\Gam_K$-equivariant functor). We can phrase this by saying that $\R$ is $(\Gam_K,G_L)$-equivariant.

Now we have a natural map
$$ \rho: \ovl{C}\left(\ovl{K}\right) \lrar V(X(C)) $$
which sends the singular points of $\ovl{C}(\ovl{K})$ to their corresponding vertices in $\ovl{\Pi} \subseteq V(X(C))$ and the smooth points to their (uniquely defined) components in $\ovl{\Lam} \subseteq V(X(C))$. The functor $\R$ introduced in the proof of Proposition~\ref{p:equivalence} induces an isomorphism of pro-finite sets
$$  \R^*:\Iso\left(F_x,F_y\right) \x{\simeq}{\lrar} \Iso\left(F_{\rho(x)},F_{\rho(y)}\right) $$
which is compatible with composition. Hence we obtain a fully-faithful functor of pro-finite groupoids
$$ \vphi: \Pi_1\left(\ovl{C}\right) \x{\simeq}{\lrar} \what{\Pi}_1(X(C)) $$
which is an equivalence since $\what{\Pi}_1(X(C))$ is connected. Finally, the action of $G_L = \Gal(L/K)$ on $\E(X(C))$ induces a natural action of $G_L$ on $\what{\Pi}_1(X)$ and $\vphi$ is naturally $(\Gam_K,G_L)$-equivariant. In particular, every $K$-rational point of $C$ will be mapped by $\rho$ to a $G_L$-fixed vertex, which in turn induces a section of the form
$$ \xymatrix{
1 \ar[r] & \what{\Pi}_1(X(C),v) \ar[r] & Q\left(\what{\Pi}_1(X(C)),G_L,v\right) \ar[r] & G_L \ar[r]\ar@/_1.8pc/[l] & 1 \\
}.$$
Similarly, for each completion $K_\nu$ of $K$ and each $K_\nu$-point $x \in C(K_\nu)$ we get a vertex $\rho(K_\nu)$ which is fixed by the corresponding decomposition group $D_\nu \subseteq G_L$. Hence each $K_\nu$-point induces a section
$$ \xymatrix{
1 \ar[r] & \what{\Pi}_1(X(C),v) \ar[r] & Q\left(\what{\Pi}_1(X(C)),D_\nu,v\right) \ar[r] & D_\nu \ar[r]\ar@/_1.8pc/[l] & 1 \\
}.$$

Now since $\vphi$ is $(\Gam_K, G_L)$-equivariant it induces a map between the Grothendieck exact sequences of $\Pi_1\left(\ovl{C}\right)$ and $\what{\Pi}_1(X(C))$. Namely, we have a map of short exact sequences
\begin{equation}\label{e:map-ses}
\xymatrix{
1 \ar[r] & \Pi_1\left(\ovl{C},x\right) \ar[r]\ar[d] & Q\left(\Pi_1\left(\ovl{C}\right),\Gam_K,x\right) \ar[r]\ar[d] & \Gam_K \ar[r]\ar@{->>}[d] & 1 \\
1 \ar[r] & \what{\Pi}_1(X(C),v) \ar[r] & Q\left(\what{\Pi}_1(X(C)),G_L,v\right) \ar[r] & G_L \ar[r] & 1}
\end{equation}
Similarly, for each place $\nu$ of $K$ we get a corresponding local version of the above diagram, namely
\begin{equation}\label{e:map-ses-loc}
\xymatrix{
1 \ar[r] & \Pi_1\left(\ovl{C},x\right) \ar[r]\ar^{\simeq}[d] & Q\left(\Pi_1\left(\ovl{C}\right),\Gam_\nu,x\right) \ar[r]\ar[d] & \Gam_\nu \ar[r]\ar@{->>}[d] & 1 \\
1 \ar[r] & \what{\Pi}_1(X(C),v) \ar[r] & Q\left(\what{\Pi}_1(X(C)),D_\nu,v\right) \ar[r] & D_\nu \ar[r] & 1}
\end{equation}
and each $K_\nu$-point of $C$ then induces a conjugacy class of compatible sections of the top and bottom row of~\ref{e:map-ses-loc}. We will use the following terminology:

\begin{define}
We will say that a conjugacy class of sections $[s] \in \Sec\left(\Pi_1\left(\ovl{C},x\right),\Gam_K\right)$ is \textbf{locally realizable} if for each place $\nu$ the restriction $[s|_{\Gam_\nu}]$ comes from a $K_\nu$-point of $C$. Similarly we will say that conjugacy class of sections $[s] \in \Sec\left(\what{\Pi}_1(X(C),v),G_L\right)$ is locally realizable if for each place $\nu$ the restriction $[s|_{D_\nu}]$ comes from a $K_\nu$-point of $C$.
\end{define}

We are now ready to formulate finite descent obstruction on $C$ in terms of sections for the Grothendieck exact sequence of $X(C)$:
\begin{prop}\label{p:har-stix-2}
Let $C$ be a geometrically connected transverse conical curve. Then $C^{\fin}(\A) \neq \emptyset$ if and only if there exists a locally realizable element in $\Sec\left(\what{\Pi}_1(X(C),v),G_L\right)$.
\end{prop}
\begin{proof}

We will rely on the following theorem which appears in~\cite{hs}:
\begin{thm}[(Harari,Stix)]\label{t:har-stix}
Let $X$ be a (not necessarily smooth) projective geometrically connected variety over $K$. Then $X^{\fin}(\A) \neq \emptyset$ if and only if there exists a locally realizable element in $\Sec\left(\Pi_1\left(\ovl{X},x\right),\Gam_K\right)$.
\end{thm}

Since the left most vertical map in~\ref{e:map-ses} is an isomorphism we get that the right square of~\ref{e:map-ses} is a pullback square, i.e. it induces an isomorphism
$$ Q\left(\Pi_1\left(\ovl{C}\right),\Gam_K,x\right) \cong Q\left(\what{\Pi}_1(X(C)),G_L,v\right) \times_{G_L} \Gam_K .$$
This means that sections of the form
$$ \xymatrix{
1 \ar[r] & \Pi_1\left(\ovl{C},x\right) \ar[r] & Q\left(\Pi_1\left(\ovl{C}\right),\Gam_K,x\right) \ar[r] & \Gam_K \ar@/_1pc/[l] \ar[r] & 1 \\
}$$
are in 1-to-1 correspondence with lifts
\begin{equation}\label{e:lift}
\xymatrix{
& & & \Gam_K \ar@{->>}[d] \ar@{-->}_{\psi}[dl] & \\
1 \ar[r] & \what{\Pi}_1(X(C),v) \ar[r] & Q\left(\what{\Pi}_1(X(C)),G_L,v\right) \ar[r] & G_L \ar[r] & 1 \\
}
\end{equation}

Now by pre-composing with $\Gam_K$ we can transform any section of the bottom row of~\ref{e:map-ses} into a lift as in~\ref{e:lift} and hence a section of the top row of~\ref{e:map-ses}. Since pre-composing commutes with conjugation by $\what{\Pi}_1(X(C),v)$ this operation induces a map on the corresponding \textbf{conjugacy classes}, i.e., we have a map
\begin{equation}\label{e:inj-map} 
\rho: \Sec\left(\what{\Pi}_1(X(C),v),G_L\right) \lrar \Sec\left(\Pi_1\left(\ovl{C},x\right),\Gam_K\right) \end{equation}
Since the map $\Gam_K \lrar G_L$ is surjective it follows that the map~\ref{e:inj-map} is injective (note that descending to conjugacy classes preserves injectivity). This injectivity implies that an element in $\Sec\left(\what{\Pi}_1(X(C),v),G_L\right)$ is locally realizable if and only if its image in $\Sec\left(\Pi_1\left(\ovl{C},x\right), \Gam_K\right)$ is locally realizable.

In order to relate Theorem~\ref{t:har-stix} to our desired result we will prove the following:
\begin{lem}\label{l:chebotarev}
If a conjugacy class $[s] \in \Sec\left(\Pi_1\left(\ovl{C},x\right),\Gam_K\right)$ is locally realizable then $[s] = \rho([t])$ for some element $[t] \in \Sec\left(\what{\Pi}_1(X(C),v),G_L\right)$ (which by the above is also locally realizable).
\end{lem}
\begin{proof}
The section $s$ corresponds to a lift $\psi: \Gam_K \lrar Q\left(\what{\Pi}_1(X(C)),G_L,v\right)$ as in~\ref{e:lift}. Let $\Gam_L \subseteq \Gam_K$ be the kernel of the quotient map $\Gam_K \lrar G_L$. Since $[s]$ is locally realizable it follows that $[s|_{\Gam_\nu}]$ comes from a class in $\Sec\left(\what{\Pi}_1(X(C),v),D_\nu\right)$ for each $\nu$. This means that $\psi$ \textbf{vanishes} when restricted to each decomposition subgroup $\Gam_\nu \cap \Gam_L$ (note that such a vanishing is invariant under conjugation). From the Chebotarev's density theorem it follows that $\psi$ vanishes on $\Gam_L$ and so descends to a section $t:G_L \lrar Q\left(\what{\Pi}_1(X(C)),G_L,v\right)$. This finishes the proof of the Lemma.
\end{proof}

From Lemma~\ref{l:chebotarev} we get that the subset of $\Sec\left(\Pi_1\left(\ovl{C},x\right),\Gam_K\right)$ of locally realizable elements can be identified with the subset of $\Sec\left(\what{\Pi}_1(X(C),v),G_L\right)$ of locally realizable elements. This means that Theorem~\ref{t:har-stix} implies Proposition~\ref{p:har-stix-2}.
\end{proof}

We are now ready to prove the main result of this section:
\begin{thm}\label{t:main}
Let $K$ be a number field and $C/K$ a geometrically connected transverse conical curve for which $C(\A_K)^{\fin} \neq \emptyset$. Then $C(K) \neq \emptyset$.
\end{thm}
\begin{proof}
Let $L/K$ be the splitting field of $X(C)$ and $G_L = \Gal(L/K)$. According to Proposition~\ref{p:har-stix-2} there exists a section
$$ \xymatrix{
1 \ar[r] & \what{\Pi}_1(X,v) \ar[r] & Q\left(\what{\Pi}_1(X),G_L,v\right) \ar[r] & G_L \ar@/_1.5pc/_{s}[l] \ar[r] & 1 \\
}.$$
which is locally realizable.

According to Theorem~\ref{t:sec-conj} the graph $X(C)$ has a fixed vertex which induces the section $s$. If this vertex corresponds to a singular point then $C$ has a $K$-rational point and we are done. Hence we can assume that this fixed vertex corresponds to a $\Gam_K$-invariant irreducible component $C_0$ of $C$. We claim that $C_0$ must have points everywhere locally. Since $C_0$ is a conic this will imply that $C_0(K) \neq \emptyset$.

Assume $C_0$ does not have points everywhere locally and let $\nu$ be a place such that $C_0(K_\nu) = \emptyset$. This means that no neighboring vertices of $C_0$ in $X(C)$ are fixed by $D_\nu$ and so $C_0$ is an \textbf{isolated vertex} in the $D_\nu$-fixed subgraph $X^{D_\nu}$. Since the section $s|_{D_\nu}$ comes from a $K_\nu$-point and this $K_\nu$-point can't be on $C_0$ we get that it has to come from a vertex on a \textbf{different} connected component of $X(C)^{D_\nu}$. This means that $s|_{D_\nu}$ is induced by two different vertices which lie on two different connected component $X(C)^{D_\nu}$. But this is impossible in view of Theorem~\ref{t:sec-conj} and we are done.
\end{proof}


%


\begin{thebibliography}{BLR90}

\bibitem[BLR90]{blr} S. Bosch, W. L\"utkebohmert and M. Raynaud, \emph{N\'eron models},
Ergebnisse der Mathematik und Ihrer Grenzgebiete, Springer (1990).

\bibitem[D89]{d} P. Deligne, \emph{Le Groupe Fondamental de la Droite Projective Moins Trois Points}, Mathematical Sciences Research Institute Publications, \textbf{16} (1989), p. 79--297.


\bibitem[SH13]{sh}
A. Skorobogatov and Y. Harpaz, \emph{Singular curves and \'etale-Brauer obstruction for surfaces}, submitted.

\bibitem[ZM89]{zm}
P. Zalesski\^{i} and O. Mel'nikov, \emph{Subgroups of profinite groups acting on
trees}, Mathematics of the USSR. Sbornik, \textbf{135} (1989), p. 405--424.


\bibitem[HS12]{hs}
Harari, D., Stix, J., \emph{Descent obstruction and fundamental exact
sequence}, J. Stix (Ed.), The arithmetic of fundamental groups, Contributions in Mathematical and Computational Science 2, Springer (2012).

\end{thebibliography}
\end{document}